\documentclass{amsart}

\usepackage[a4paper,margin=.75in]{geometry}
\usepackage{mathpazo,newpxtext}
\usepackage{mathtools,amssymb}
\usepackage{tabularx}
\usepackage[hidelinks]{hyperref}
\usepackage[nameinlink]{cleveref}
\usepackage{tikz}
\usepackage{microtype}

\newtheorem{definition}{Definition}[section]
\newtheorem{conjecture}{Conjecture}[section]
\newtheorem{corollary}[conjecture]{Corollary}
\newtheorem*{assumption}{Assumption}
\theoremstyle{remark}
\newtheorem{remark}{Remark}
\numberwithin{equation}{section}

\def\P{\ensuremath{\mathbb{P}}}
\newcommand{\Q}[1][]{\ensuremath{\mathbb{Q}(#1\,)}}

\begin{document}

\title{Conjectures on the distribution behavior of the class numbers of certain real quadratic number fields}
\author{Jinwen XU}
\date{April 11, 2021}

\begin{abstract}
    Given a random real quadratic field from \( \{ \Q[\sqrt{p}] ~|~ p \text{ primes} \} \), the conjectural probability \( \P(h=q) \) that it has class number \( q \) is given for all positive odd integers \( q \). Some related conjectures in \cite{CL83}, known as the Cohen-Lenstra heuristic, are given here as corollaries. These results suggest that the set of real quadratic number fields may have some natural hierarchical structures.
\end{abstract}

\maketitle

In this article, we shall study the distribution behavior of the class numbers of real quadratic number fields in the set \( \mathfrak{P} \coloneqq \{ \Q[\sqrt{p}] ~|~ p \text{ primes} \} \). These class numbers are all odd integers according to \cite[Sec.3.8~(Theorem~8)]{BS66}. Let \( \P(h=q) \) stand for the probability that a randomly chosen member from \( \mathfrak{P} \) has class number \( q \), where \( q \) is any positive odd integer. In \cref{sec:conjecture}, the conjectural value of \( \P(h=q)/\P(h=1) \) is given in \cref{conj:main}. Then in \cref{sec:corollaries}, \( \P(h=1) \) is calculated based on \cref{conj:main}, alongside with some corollaries. Finally, the current status of this conjecture and some ideas for future work, including some similar observations to the set \( \mathfrak{A} \coloneqq \{ \Q[\sqrt{d}] ~|~ d > 0 \text{ square free} \} \), are given in the last section.

\section{The main conjecture}\label{sec:conjecture}

First we define \( \P(h=q) \) precisely, and make an important assumption.

\begin{definition}
    Denote the \( i \)-th prime number as \( p_i \). Define
    \[
        \P(h=q) \coloneqq \lim_{x\to\infty} \frac{\#\{i ~|~ h(\Q[\sqrt{p_i}])=q,~ i\leqslant x\}}{x}.
    \]
\end{definition}

\begin{assumption}
    \( \P(h=q) \) exists for all positive odd integers \( q \).
\end{assumption}

For convenience, denote \( \P(h=1) \) by \( \P_0 \), and write 
\[
    \P(h=q) = \lambda_q\P_0,
\]
where \( \lambda_q \) is a constant depending on \( q \). In this setting, one has \( \lambda_1 = 1 \).

\bigskip
Now we can state the main conjecture, which gives the conjectural value of \( \lambda_q \):

\begin{conjecture}\label{conj:main}
    \begin{enumerate}
        \item For any odd prime number \( p \), one has:
        \begin{align*}
            \lambda_p &= \frac{1}{p(p-1)}\\
            \lambda_{p^{n}} &= \lambda_{p^{n-1}} (p^{\lfloor n/2+1 \rfloor}-1)^{-1},\quad n \geqslant 2.
        \end{align*}
        \item For \( q = p_1^{r_1}\cdots p_t^{r^t} \), one has:
        \[
            \lambda_q = \lambda_{p_1^{r_1}}\cdots \lambda_{p_t^{r_t}}.
        \]
    \end{enumerate}
\end{conjecture}

\begin{remark}
    \( \lambda \) is an arithmetic function in \( q \). Part (2) of \cref{conj:main} says that \( \lambda \) is multiplicative.
\end{remark}
\begin{remark}
    Part (2) of \cref{conj:main} is actually a result of the assumption combined with part (1):
    \begin{align*}
        \P(h=q) &= \left(\prod_{m=1}^t \P(\{p_m^{r_m}\mid h\}\wedge \{p_m^{r_m+1}\nmid h\})\right)\P(p\nmid h \text{~for all other~}p)\\
        &= \left(\prod_{m=1}^t \frac{\P(h=p_m^{r_m})}{\P(p_m\neq p'\nmid h)}\right)\P(p\nmid h \text{~for all other~}p)\\
        &= \left(\prod_{m=1}^t \frac{\lambda_{p_m^{r_m}}\P(p\nmid h \text{~for all~}p)}{\P(p_m\neq p'\nmid h)}\right)\P(p\nmid h \text{~for all other~}p)\\
        &= \left(\prod_{m=1}^t \lambda_{p_m^{r_m}}\P(p_m\nmid h)\right)\P(p\nmid h \text{~for all other~}p)= \left(\prod_{m=1}^t \lambda_{p_m^{r_m}}\right)\P_0.
    \end{align*}
\end{remark}

Below is a table illustrating \cref{conj:main}:
\begin{center}
    \small
    \begin{tabularx}{380pt}{c|X|c|X|c|X|c|X|c|X}
        \( q \) & \footnotesize\( \P(h=q)/\P_0 \) &  \( q \) & \footnotesize\( \P(h=q)/\P_0 \) &  \( q \) & \footnotesize\( \P(h=q)/\P_0 \) &  \( q \) & \footnotesize\( \P(h=q)/\P_0 \) &  \( q \) & \footnotesize\( \P(h=q)/\P_0 \) \\\hline
        1 & 1    & 11 & 1/110 & 21 & 1/252 & 31 & 1/930 & 41 & 1/1640 \\
        3 & 1/6  & 13 & 1/156 & 23 & 1/506 & 33 & 1/660 & 43 & 1/1806 \\
        5 & 1/20 & 15 & 1/120 & 25 & 1/480 & 35 & 1/840 & 45 & 1/960 \\
        7 & 1/42 & 17 & 1/272 & 27 & 1/384 & 37 & 1/1332& 47 & 1/2162 \\
        9 & 1/48 & 19 & 1/342 & 29 & 1/812 & 39 & 1/936 & 49 & 1/2016
    \end{tabularx}
\end{center}
\medskip

In the next section, we shall calculate \( \P_0 \) based on this conjecture.

\section{Calculation of \texorpdfstring{\( \P_0 \)}{P\_0} and some corollaries}\label{sec:corollaries}

Since \( \P_0 = \prod_{p\geqslant 3} \P(p \nmid h) \), where \( p \) runs over all odd prime numbers, we only need to know \( \P(p \nmid h) \).

Fix a prime number \( p \). One clearly has
\[
    \P(h\in\{p^k\}) = \P(p \mid h) \P(p \neq p' \nmid h).
\]
On the other hand, according to \cref{conj:main},
\[
    \P(h\in\{p^k\}) = \sum_k \lambda_{p^k} \P_0 = \sum_k \lambda_{p^k} \P(p \nmid h) \P(p \neq p' \nmid h).
\]
Thus, 
\begin{equation}\label{eq:sum1}
    \sum_k \lambda_{p^k} = \frac{\P(p \mid h)}{\P(p \nmid h)} = \frac{1 - \P(p \nmid h)}{\P(p \nmid h)}
\end{equation}

Now we write \( \sum_k \lambda_{p^k} \) in another way. Beginning with the equality\footnote{The author was not able to find a reference for this, but it can be proved simply by checking that the ratio of L.H.S and R.H.S tends to 1.}
\[
    \prod_{k\geqslant 1} (1-p^{-k})^{-1} = 1+\frac{1}{p-1}\bigg(1+\frac{1}{p-1}\bigg(1+\frac{1}{p^2-1}\bigg(1+\frac{1}{p^2-1}\bigg(1+\frac{1}{p^3-1}\bigg(1+\frac{1}{p^3-1}\bigg(1+\cdots\bigg),
\]
one has:
\begin{align*}
    &\prod_{k\geqslant 1} (1-p^{-k})^{-1} - \frac{p}{p-1} = \bigg(\frac{1}{p-1}\bigg)^2\bigg(1+\frac{1}{p^2-1}\bigg(1+\frac{1}{p^2-1}\bigg(1+\frac{1}{p^3-1}\bigg(1+\frac{1}{p^3-1}\bigg(1+\cdots\bigg)\\
    \iff &
    \prod_{k\geqslant 2} (1-p^{-k})^{-1} - 1 = \frac{1}{p(p-1)}\bigg(1+\frac{1}{p^2-1}\bigg(1+\frac{1}{p^2-1}\bigg(1+\frac{1}{p^3-1}\bigg(1+\frac{1}{p^3-1}\bigg(1+\cdots\bigg),
\end{align*}
which is precisely
\begin{equation}\label{eq:sum2}
    \frac{1-\prod_{k\geqslant 2} (1-p^{-k})}{\prod_{k\geqslant 2} (1-p^{-k})} = \sum_k \lambda_{p^k}.
\end{equation}
Compare \cref{eq:sum1,eq:sum2}, one gets
\begin{corollary}\label{coro:nmid}
    The conjectural probability that an odd prime number \( p_0 \) does not divide the class number of \( \Q[\sqrt{p}] \) is
    \[
        \P(p_0 \nmid h) = \prod_{k\geqslant 2} (1-p_0^{-k}).
    \]
\end{corollary}\noindent
And thus finally
\begin{corollary}\label{coro:p0}
    The conjectural probability of \( \Q[\sqrt{p}] \) has class number one is
    \begin{align*}
        \P_0 &= \prod_{\underset{p\text{~prime}}{p\geqslant 3}}\prod_{k\geqslant 2} (1-p^{-k})\\
        &= \prod_{k\geqslant 2} \left((1-2^{-k})\zeta(k)\right)^{-1} \\
        &\approx 0.75446.
    \end{align*}
\end{corollary}
\Cref{coro:nmid,coro:p0} coincide with conjectures (C7) and (C11) in \cite{CL83}. Here we arrived at those results from a different starting point.

\section{Status of the conjecture and some remarks for further work}

\subsection{Current status}

\Cref{conj:main} says that over 75\% of real quadratic fields in \( \{ \Q[\sqrt{p}] ~|~ p \text{ primes} \} \) have class number one, thus implies Gauss's class number conjecture for real quadratic fields, which states that there are infinitely many such fields with class number one. We are far from proving this yet. Numerical calculations dating back to late 1980s have showed that the actual ratio tends to \( \P_0 \) very slowly, see \cite[\textsc{Figure~4.1}]{SW88}.

One related known result is that the number of real quadratic fields with absolute discriminant \( \leqslant X \) and class number divisible by \( q \) is \( \gg X^{(1/q)-\varepsilon} \) for all \( \varepsilon>0 \) (Yu, \cite{Yu02}); better estimation for small class number cases can be found in \cite{Byeon03,Byeon06}.

\subsection{On the set \texorpdfstring{\( \mathfrak{A} \)}{A} of all real quadratic fields}

In this subsection, we briefly discuss some observations on the class numbers of members in \( \mathfrak{A} = \{ \Q[\sqrt{d}] ~|~ d > 0 \text{ square free} \} \). Similarly to the previous discussion, define \( \P_i \) to be the probability that a randomly chosen member from \( \mathfrak{A} \) has class number \( i \), and assume that \( \P_i \) exists for all \( i \geqslant 1 \). Then numerical result suggests that:
\begin{conjecture}\label{conj:all-odd}
    For positive odd integers \( q \), \( \P_q / \P_1 = \lambda_q \).
\end{conjecture}
\begin{conjecture}\label{conj:all-relation}
    For any even number \( n > 0 \), write \( n = 2^r q \), where \( q \) is an odd number. Then \( \P_n = \lambda_q \P_{2^r} \).
\end{conjecture}
With the conjectures above, one can deduce that:
\begin{corollary}
    \( \P_0 = \sum_{k=0}^{\infty} \P_{2^k}. \)
\end{corollary}
\begin{proof}[``Proof'']
    \leavevmode\vspace{-1.4\baselineskip}
    \begin{align*}
        \hspace{3.5cm}&1-(\P_1+\P_2+\P_4+\P_8+\P_{16}+\cdots)\\
        =~& (\lambda_3+\lambda_5+\lambda_7+\lambda_9+\cdots)(\P_1+\P_2+\P_4+\P_8+\P_{16}+\cdots)
        &(\text{by \cref{conj:all-odd} and \ref{conj:all-relation}})\\
        =~& (1/\P_0 - 1)(\P_1+\P_2+\P_4+\P_8+\P_{16}+\cdots)
        &(\text{by \cref{conj:main}})
    \end{align*}
    Thus \( (1/\P_0)(\P_1+\P_2+\P_4+\P_8+\P_{16}+\cdots)=1 \), hence \( \P_0 = \sum_{k=0}^{\infty} \P_{2^k} \).
\end{proof}

Despite the (conjectural) summation formula above, the relationship among \( \P_1 \), \( \P_2 \), \( \P_4 \), \( \P_8 \), etc. is currently unknown. However, from the numerical calculation, it seems that \( \P_2 > \P_1 \), which suggests that properties of the even part of class groups might be different from those of the odd part.

\subsection{A heuristic assertion}

According to \cref{conj:main}, for any odd integers \( q_1,q_2 \) with \( q_1 \mid q_2 \), the ratio \( \lambda_{p_1}/\lambda_{p_2} \) is an integer. This suggests that the set \( \mathfrak{P} = \{ \Q[\sqrt{p}] ~|~ p \text{ primes} \} \) may have some natural hierarchical structures.

\subsubsection{An illustration}
Consider \( \lambda_{1}/\lambda_{3} = 6 \), it is reasonable to think of this as if for any member in \( \mathfrak{P} \) with class number three, there are six members of class number one ``growing'' out of it, as illustrated below:
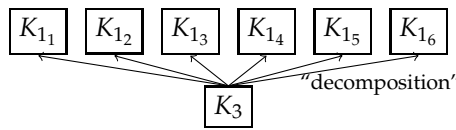
\begin{center}
    \begin{tikzpicture}
        \foreach \i in {1,...,6}{
            \node at (\i,0) {\fbox{\( K_{1_{\i}} \)}};
            \draw[->] (3.5,-0.72) -- (\i,-0.32);
        }
        \node at (3.5,-1) {\fbox{\( K_3 \)}};
        \node at (5.5,-0.7) {\footnotesize ``decomposition''};
    \end{tikzpicture}
\end{center}
Here \( K_3 \) represents some real quadratic field in \( \mathfrak{P} \) with class number three, and \( K_{1_{i}} \) represents those with class number one that is related to \( K_3 \) in this hypothetical hierarchical structure.

If one regards class number as a measurement of the distance from the ring of integers in this number field to an UFD, then one may interpret the meaning of this hierarchical structure as: the obstacles for the integer factorization in \( K_3 \) are contributed by those \( K_{1_{i}} \) (which are all closer to an UFD) related to it. More specifically, this means that the fundamental information about factorization contained in the inner structures of \( K_{1_{i}} \) are somehow kept and embedded in that of \( K_3 \). Thus, for example, one may expect the classical notion of \emph{ideals} to be part of this new hypothetical structure.

The example above demonstrates only a small piece in the hierarchical structure. The global picture would be like patching together infinitely many such pieces with 1 and 3 replaced by \( q_1 \mid q_2 \). In this process, as composite numbers such as 27 (of form \( p^k \)), 35 (of form \( pq \)) and the others coming into the picture, an intriguing global structure emerges, for which \cref{conj:main} only describes the general graph-theoretic information. 

\subsubsection{The ``eigenvalue''}
On the other hand, thinking along this direction, the equality, or at least some quantities in the equality:
\begin{align*}
    &\bigg( 1+\frac{1}{p-1} \bigg)
    \bigg( 1+\frac{1}{p^2-1} \bigg)
    \bigg( 1+\frac{1}{p^3-1} \bigg)
    \cdots\\
    =\quad&
    1+\frac{1}{p-1}\bigg(1+\frac{1}{p-1}\bigg(1+\frac{1}{p^2-1}\bigg(1+\frac{1}{p^2-1}\bigg(1+\frac{1}{p^3-1}\bigg(1+\frac{1}{p^3-1}\bigg(1+\cdots\bigg)
\end{align*}
may be seen as the ``signature'' or ``eigenvalue'' of this particular hierarchical structure corresponding to \( \mathfrak{P} \), encoding its degree information. Previously we have only considered the set \( \mathfrak{P} \), but one can reasonably speculate that this kind of hierarchical structure exists in many other context, and each corresponds to a beautiful equality of its own. This suggests, from another point of view, that the study of it might be fruitful.

\subsubsection{On the construction}
To really construct such a structure, \cite[(C11)]{CL83} suggests that the order at each vertex may have something to do with the order of the corresponding automorphism groups, which offers a clue to get started. Currently we do not know much about the arrangement of the vertices or when shall two of them connected, and the above heuristic assertion on the hierarchical structure may still be a little over-simplified. However, once this hierarchical structure is successfully constructed and its properties established, \cref{conj:main} follows immediately (although the proof itself would appear to be non-constructive).

It should be pointed out, however, that the set of complex quadratic fields does not seem to share a similar hierarchical structure, since the class number \( h\to\infty \) as the discriminant tends to infinity. The author believes that it is due to the presence of \( \sqrt{-1} \) that somehow causes the corresponding hierarchical structure ``collapse''.

\subsection{Final remarks}

All these conjectures are obtained by merely observing several small class number tables of size \( \sim 10^6 \). Thus it is possible that the conjectural value \( \lambda_{p^k} \) for large \( k \) is wrong (the exponent \( \lfloor n/2+1 \rfloor \) seems to be somewhat mysterious). However, the author believes that even if there are some flaws in the conjectures, the heuristic assertion above should still make some sense.

\end{document}